\documentclass[12pt]{amsart}

\usepackage{amssymb,amsmath,amsthm,amsxtra,hyperref}
\usepackage{tikz}

\newcommand{\cA}{\mathcal{A}}
\newcommand{\BB}{\mathcal{B}}
\newcommand{\Dunb}{D_{\mathrm{unb}}}
\newcommand{\Ddeg}{D_{\mathrm{deg}}}
\newcommand{\EE}{\mathcal{E}}
\newcommand{\FF}{\mathcal{F}}
\newcommand{\GG}{\mathbb{G}}
\newcommand{\mm}{\mathfrak{m}}
\newcommand{\Moo}{\overline{\mathcal{M}}_{0,0}}
\newcommand{\OO}{\mathcal{O}}
\newcommand{\PP}{\mathbb{P}}
\newcommand{\QQ}{\mathbb{Q}}
\newcommand{\ZZ}{\mathbb{Z}}

\DeclareMathOperator{\coker}{coker}
\DeclareMathOperator{\Mor}{Mor}
\DeclareMathOperator{\Pic}{Pic}
\DeclareMathOperator{\rank}{rank}
\DeclareMathOperator{\Supp}{Supp}

\theoremstyle{plain}
\newtheorem{theorem}{Theorem}[section]
\newtheorem{lemma}[theorem]{Lemma}

\newtheorem{proposition}[theorem]{Proposition}

\theoremstyle{definition}

\begin{document}

\mbox{}
\vspace{-1.1ex}
\title[Effective cone of the Quot scheme compactifying $\Mor_d(\PP^1,\text{Grassmannian})$]{The effective cone of the space of parametrized rational curves in a Grassmannian}
\author{Shin-Yao Jow}
\address{Department of Mathematics,
University of Pennsylvania,
Philadelphia, PA 19104}
\email{\texttt{jows@math.upenn.edu}}
\date{13 July 2011}

\begin{abstract}
We determine the effective cone of the Quot scheme parametrizing all rank~$r$, degree~$d$ quotient sheaves of the trivial bundle of rank~$n$ on $\PP^1$. More specifically, we explicitly construct two effective divisors which span the effective cone, and we also express their classes in the Picard group in terms of a known basis.
\end{abstract}


\maketitle

\section{Introduction} 

Let $V$ be an $n$-dimensional vector space over an algebraically closed ground field, and let $V_{\PP^1}=\OO_{\PP^1}\otimes V$ be the trivial bundle of rank~$n$ on $\PP^1$. In this paper, we determine the effective cone of the Quot scheme $R$ parametrizing all degree~$d$, rank~$r$ quotient sheaves of $V_{\PP^1}$.

The Quot scheme $R$ is closely related to the Kontsevich moduli spaces of genus-zero stable maps to Grassmannians. Indeed let $\GG$ be the Grassmannian of $r$-dimensional quotient spaces of $V$. Then $R$ is a compactification of $R^0=\Mor_d(\PP^1,\GG)$, the space of all degree~$d$ morphisms from $\PP^1$ to $\GG$, or equivalently, the space of all degree~$d$, rank~$r$ quotient bundles of $V_{\PP^1}$. One can also consider Quot schemes over more general curves other than $\PP^1$, and like the Kontsevich moduli spaces, they have been used in enumerative geometry (see for example \cite{Ber94}, \cite{Ber97}, \cite{BDW}, \cite{RRW}, \cite{Ram06}, \cite{Ram09}).

Many basic properties of the Quot scheme $R$ were established in \cite{Str} by 
Str\o mme. For example he showed that $R$ is an irreducible, rational, nonsingular projective variety of dimension $nd+r(n-r)$ \cite[Theorem~2.1]{Str}. It is not too difficult to see that if $r=n-1$ then $R$ is a projective space, and if $d=0$ then $R$ is just the Grassmannian $\GG$, so in these cases $\Pic R\cong \ZZ$ \cite[Proposition~6.1]{Str}. In all other cases, i.e. $0\le r \le n-2$ and $d\ge 1$, Str\o mme proved that $\Pic R\cong \ZZ^2$, and he gave generators for the nef cone of $R$ \cite[Theorem~6.2]{Str}. More specifically, there is a universal short exact sequence on $R\times \PP^1$: \[
 0 \to \cA \to V_{R\times \PP^1} \to \BB \to 0. \]
For each point $p\in R$, $\BB_p=\BB|_{\{p\}\times \PP^1}$ is the degree~$d$, rank~$r$ quotient sheaf of $V_{\PP^1}$ represented by $p$. Denote by $\pi_1$ and $\pi_2$ the two projection maps from $R\times \PP^1$ to $R$ and $\PP^1$, respectively. Let $\BB(m)=\BB\otimes \pi_2^*\OO_{\PP^1}(m)$, and let $B_m={\pi_1}_*\BB(m)$. Then the divisor classes $c_1(B_{d-1})$ and $c_1(B_d)-c_1(B_{d-1})$ on $R$ form a $\ZZ$-basis for $\Pic R$ and also span the nef cone.

There is another useful $\ZZ$-basis for $\Pic R$ described by Ramirez in \cite[\S 3]{Ram09}.\footnote{Although there was a standing assumption that $n=4$ and $r=2$ throughout \cite{Ram09}, the part in \S 3 about changing basis for $\Pic R$ can be straightforwardly generalized.} Let $h=c_1\bigl(\pi_2^*\OO_{\PP^1}(1)\bigr)$. Then the divisor classes \[
 Y={\pi_1}_*\bigl(h\cdot c_1(\BB)\bigr), \quad D={\pi_1}_*\bigl(c_2(\BB)\bigr) \]
form a $\ZZ$-basis for $\Pic R$. Indeed a Grothendieck-Riemann-Roch computation \cite[Lemma~3.2]{Ram09} yields\footnote{There is a sign error in the second formula in \cite[Lemma~3.2]{Ram09}. Except for that, both formulae again hold true for any $n\ge 2$, $0\le r \le n-2$, and $d\ge 1$.} 
 \begin{gather*}
  c_1(B_d)-c_1(B_{d-1})=Y, \\
  c_1(B_{d-1})=2dY-D.
 \end{gather*}
Thus in terms of the basis $Y$ and $D$, the nef cone is spanned by $Y$ and $2dY-D$.
 
The divisors $D$ and $Y$ have geometric interpretations, at least when $r\ge 2$. To see this, let $\PP(V)$ be the projective space of one-dimensional subspaces of $V$. Each point $p\in R^0$ can be viewed as a parametrized rational scroll $\PP(\cA_p)\to \PP(V)$ of dimension $n-r$ in $\PP(V)$. Hence those points in $R^0$ which, when viewed as rational scrolls, meet a fixed $(r-2)$-dimensional subspace in $\PP(V)$, form a divisor of $R^0$, and the closure of this divisor in $R$ is $D$ (note that it was shown in \cite[Corollary~3.3.8]{Shao} that the complement of $R^0$ in $R$ is irreducible of codimension $r$). Similarly, those points in $R^0$ for which, when viewed as rational scrolls, the fibers over a fixed point $0\in \PP^1$ meet a fixed $(r-1)$-dimensional subspace in $\PP(V)$, form a divisor of $R^0$, and $Y$ is the closure of this divisor in $R$.

In order to describe the effective cone of $R$, we construct two effective divisors $\Dunb$ and $\Ddeg$, and then use test curves to show that they span the effective cone, as well as to express their classes in $\Pic R$ in terms of the basis $Y$ and $D$. The definitions of $\Dunb$ and $\Ddeg$ were inspired by the work of Coskun-Starr \cite{CS} on the effective cone of the Kontsevich moduli space $\Moo(\GG,d)$. They did not have a complete description of the cone when $r<d$, but they showed that if $n-r$ and $d$ are fixed, then the effective cone of $\Moo(\GG,d)$ grows as $r$ increases, and stabilizes when $r\ge d$. They then went on to show that the stable effective cone, i.e. the effective cone of $\Moo(\GG,d)$ when $d=r$, is spanned by  the boundary divisors and two more effective divisors which they named $\Dunb$ and $\Ddeg$ (the subscripts stand for ``unbalanced'' and ``degenerate'', respectively). To state the definitions, let $k=n-r$ be the rank of the universal subsheaf $\cA$. The definition of $\Dunb$ depends on whether $k$ divides $d$. When $k\mid d$, we define $\Dunb$ in the same way as Coskun-Starr: \[
 \Dunb=\{p\in R \mid \text{$\cA_p$ has unbalanced splitting (as a locally free sheaf on $\PP^1$)} \}. \]
When $k\nmid d$, set $d=kd_1+(k-\ell_1)$ for $0<\ell_1<k$. Then for all $p$ outside a codimension two locus in $R$, $\cA_p$ splits as a direct sum of $\ell_1$ copies of $\OO_{\PP^1}(-d_1)$ and $(k-\ell_1)$ copies of $\OO_{\PP^1}(-d_1-1)$. In this situation the direct sum of the $\ell_1$ copies of $\OO_{\PP^1}(-d_1)$ is a distinguished subsheaf $\EE_p$ of $\cA_p$, whose projectivization $\PP(\EE_p)\to \PP(V)$ is the directrix of the scroll $\PP(\cA_p)\to \PP(V)$. We define $\Dunb$ to be the closure in $R$ of the locus of $p$ whose corresponding directrix $\PP(\EE_p)\to \PP(V)$ meets a fixed $(n-2-\ell_1)$-dimensional subspace in $\PP(V)$. Note that this is different from the definition of Coskun-Starr in \cite{CS}, which asked that the \emph{linear span} of the directrix meets a fixed subspace in $\PP(V)$ of codimension $\ell_1(d_1+1)$. We modified it because it does not work when $d\gg r$.

Coskun-Starr defined $\Ddeg$ on $\Moo(\GG,d)$, $d=r\ge 2$, to be the divisor of maps whose corresponding scrolls degenerate, namely lie in some hyperplane. This does not give a divisor as soon as $d>r$, so new approaches are required to formulate a definition of $\Ddeg$ on $R$ which will work for all $d\ge 1$. Our starting point is the observation that in the case $d=r\ge 2$ considered by Coskun-Starr, if $p\in R^0$, then the scroll corresponding to the subbundle $\cA_p$ is degenerate if and only if the quotient bundle $\BB_p$ has unbalanced splitting (Proposition~\ref{p:Ddeg}). Inspired by this, whenever $r\ge 2$ we define $\Ddeg$ to be the closure in $R$ of the locus of those $p\in R^0$ whose corresponding quotient bundles $\BB_p$ are ``unbalanced in the previously defined sense'', which means unbalanced splitting if $r\mid d$, or incidence condition on the directrix of the dual scroll $\PP(\BB^{\spcheck}_p)\to \PP(V^{\spcheck})$ if $r\nmid d$ (details are spelled out in Section~\ref{s:intersect generators}, the paragraph following Proposition~\ref{p:Ddeg}). When $r=1$ however, no $p$ in $R^0$ has unbalanced quotient bundle $\BB_p$; on the other hand, the complement of $R^0$ in $R$ becomes an irreducible divisor by \cite[Corollary~3.3.8]{Shao}. Therefore it is natural to define $\Ddeg=R\setminus R^0$ when $r=1$. Finally when $r=0$, $\BB_p$ is a torsion sheaf of degree~$d$ for every $p\in R$. Generically the support of $\BB_p$ consists of $d$ distinct points on $\PP^1$, and we define $\Ddeg$ to be the locus of $p\in R$ for which some of the $d$ points in $\Supp \BB_p$ coincide.

With all the necessary definitions in place, we can now state the main theorem.

\begin{theorem} \label{t:main}
 Let $R$ be the Quot scheme parametrizing all rank~$r$, degree~$d$ quotient sheaves of the trivial bundle of rank~$n$ on $\PP^1$. Let $k=n-r$ and assume that $k\ge 2$ and $d\ge 1$. Then the effective cone of $R$ is spanned by the two effective divisors $\Dunb$ and $\Ddeg$ defined above. Moreover, their classes in $\Pic R$ can be expressed in terms of the basis $D$ and $Y$ as
 \begin{align*}
   \Dunb &= c_1\Bigl( -D + \bigl(d+\Bigl\lceil \frac{d}{k} \Bigr\rceil \bigr)Y \Bigr), \\
   \Ddeg &= \begin{cases} \displaystyle
   c_2\Bigl(D + \bigl( -d+\Bigl\lceil \frac{d}{r} \Bigr\rceil \bigr)Y \Bigr), &\text{if\/ $r>0$;} \\
   2(d-1)Y, &\text{if\/ $r=0$,}
   \end{cases} 
  \end{align*}
for some positive $c_1$ and $c_2$. Precisely, $c_1=1$ if\/ $k\mid d$ and $r\ne 0$, while $c_1=d_1(\ell_1+1)$ if\/ $k\nmid d$ and we set\/ $d=kd_1+(k-\ell_1)$ for\/ $0<\ell_1<k$. Similarly $c_2=1$ if\/ $r\mid d$, while $c_2=d_2(\ell_2+1)$ if\/ $r\nmid d$ and we set\/ $d=rd_2+(r-\ell_2)$ for\/ $0<\ell_2<r$.
\end{theorem}

We do not know the value of $c_1$ when $r=0$ and $k\mid d$, although the natural guess is that it is $1$. The nef cone and the effective cone of $R$ are shown in  Figure~\ref{f:cones}.

\begin{figure}[htbp] 
\begin{center}
\begin{tikzpicture}
\begin{scope}[xshift=-4.1cm]
\draw (-3,0) -- (3,0) node[right]{$D$};
\draw (0,-2.7) node[below=7pt]{When $r>0$;} -- (0,3) node[above]{$Y$};

\draw (0,0) -- (-1.5,3);
\draw (-1.8,3.3) node{$-2d$};
\fill[green!20!white] (0,0) -- (-1.5,3) -- (0,3) -- cycle;
\draw (-1,2) node[above right]{Nef};
\draw (0,0) -- (-2.5,3);
\draw (-3.5,3.2) node{$-(d+\lceil \frac{d}{k} \rceil )$};
\draw (0,0) -- (3,-2) node[below]{$-d+ \lceil \frac{d}{r} \rceil $};
\fill[blue!20!white, fill opacity=0.5] (0,0) -- (-2.5,3) -- (3,3) -- (3,-2) -- cycle;
\draw (1,1) node[above]{Eff};
\end{scope}
\begin{scope}[xshift=4.1cm]
\draw (-2.7,0) -- (2.1,0) node[right]{$D$};
\draw (0,-2.7) node[below=7pt]{when $r=0$.} -- (0,3) node[above]{$Y$};

\draw (0,0) -- (-1.5,3);
\draw (-1.8,3.3) node{$-2d$};
\fill[green!20!white] (0,0) -- (-1.5,3) -- (0,3) -- cycle;
\draw (-1,2) node[above right]{Nef};
\draw (0,0) -- (-2.5,3);
\draw (-3.5,3.2) node{$-(d+\lceil \frac{d}{k} \rceil )$};
\fill[blue!20!white, fill opacity=0.5] (0,0) -- (-2.5,3) -- (0,3) -- cycle;
\draw (-0.5,1) node[above]{Eff};
\end{scope}
\end{tikzpicture}
\end{center}
\caption{The nef cone and the effective cone of $R$. Left: $r>0$. Right: $r=0$. The slopes of the boundaries of the cones are labeled.}
\label{f:cones}
\end{figure}
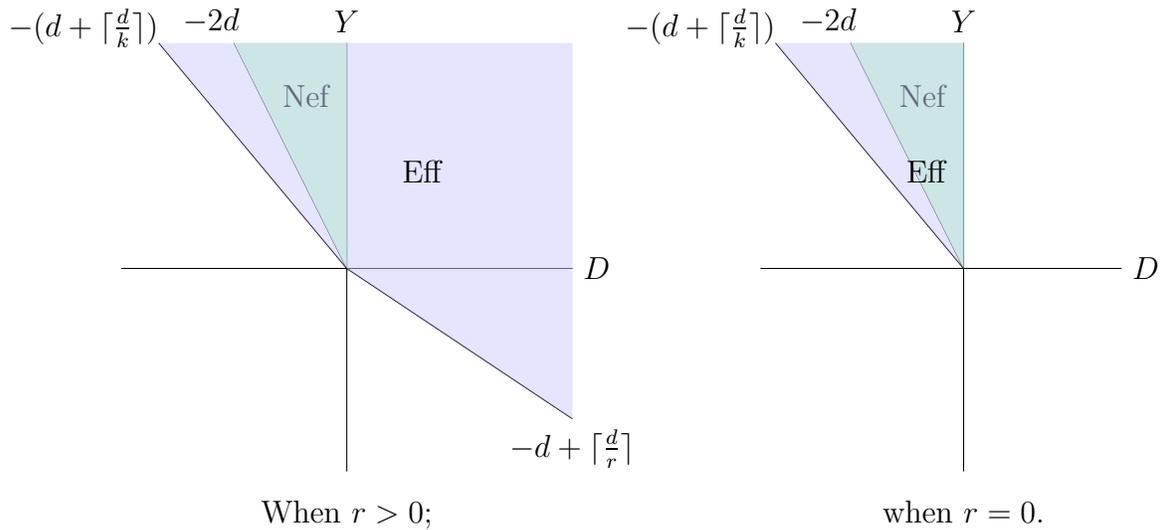

In order to prove Theorem~\ref{t:main}, we start with constructing test curves on $R$ in Section~\ref{s:test curves}. Then we compute the intersection numbers of the test curves with the divisors $D$ and $Y$ in Section~\ref{s:intersect basis}, and with $\Dunb$ and $\Ddeg$ in Section~\ref{s:intersect generators}. These intersection numbers come into the proof of Theorem~\ref{t:main} given in Section~\ref{s:proof r>0} (the case $r>0$) and Section~\ref{s:proof r=0} (the case $r=0$). It is perhaps worth mentioning that we use a somewhat nonstandard Lemma~\ref{l:span Eff} to show that the two divisors $\Dunb$ and $\Ddeg$ span the effective cone, which avoids the task of constructing moving curves in the usual strategy.

\section*{Acknowledgments}

The author first learned about this problem in the AMS 2010 Mathematics Research Communities summer conference on Birational Geometry and Moduli Spaces, which the author greatly benefited from. The author would also like to thank Izzet Coskun and Yijun Shao for answering questions about the Quot scheme.  

\section{Test curves on the Quot scheme} \label{s:test curves}

In this section, we define the test curves we will be using to intersect with the various divisors on the Quot scheme $R$. To this end, pick an arbitrary smooth projective curve $C$ and $k$ line bundles $A_1,\ldots,A_k$ on $C$ with nonzero sections. Let $a_i=\deg A_i\ge 0$ for $i=1,\ldots,k$. Let $S=C\times \PP^1$, and let  $\pi_1\colon S\to C$ and $\pi_2\colon S\to \PP^1$ be the two projection maps. If $L$ is a line bundle on $C$ and $m\in\ZZ$, we denote by $L(m)$ the line bundle $\pi_1^*L\otimes \pi_2^*\OO_{\PP^1}(m)$ on $S$. Let $A_S$ be the direct sum of the line bundles \[
 A_S=\bigoplus_{i=1}^k A_i^{\spcheck}(-m_i) \]
on $S$, where if $k\mid d$ then $m_i=d/k$ for all $i$, and if $k\nmid d$, set $d=kd_1+(k-\ell_1)$ for $0<\ell_1<k$, and choose $m_i=d_1$ for $i\le \ell_1$ and $m_i=d_1+1$ for $i>\ell_1$. (In other words, $m_1+\cdots+m_k=d$ is the ``most balanced'' partition of $d$ into $k$ parts.) Given a section $s_i$ of $A_i(m_i)\otimes V$ for each $i=1,\ldots,k$, one obtains a sheaf homomorphism \[
 \varphi\colon A_S=\bigoplus_{i=1}^k A_i^{\spcheck}(-m_i) \longrightarrow V_S=\OO_S\otimes V, \]
which is simply multiplication by $s_i$ on each summand $A_i^{\spcheck}(-m_i)$. If  $a_i$'s are sufficiently large and $s_i$'s are sufficiently general, the restriction of $\varphi$ on each fiber of $\pi_1$ is injective, so its cokernel is a degree~$d$, rank~$r$ quotient sheaf of $V_{\PP^1}$. Hence $\varphi$ induces a morphism \[
 \alpha\colon C\longrightarrow R, \] 
which is one of the test curves we will use.

When $r>0$, we can construct another curve \[
 \beta\colon C\longrightarrow R \]
which is induced by a surjective sheaf homomorphism \[
 \psi\colon V_S \longrightarrow B_S=\bigoplus_{i=1}^r B_i(n_i), \]
where $B_i$ are line bundles of degrees $b_i\ge 0$ on $C$, and $n_i=d/r$ for all $i$ if $r\mid d$, while if $r\nmid d$, write $d=rd_2+(r-\ell_2)$ for $0<\ell_2<r$, and choose $n_i=d_2$ for $i\le \ell_2$ and $n_i=d_2+1$ for $i>\ell_2$. (In other words, $n_1+\cdots+n_r=d$ is the ``most balanced'' partition of $d$ into $r$ parts.)  Since we always have $k\ge 2$, such a surjective homomorphism $\psi$ exists as long as $b_i=\deg B_i$ are sufficiently large.

When $r=0$, we can no longer construct the curve $\beta$ as above, so a separate construction is needed. This will be done in Section~\ref{s:proof r=0}.

\section{Intersections of the test curves with $D$ and $Y$} \label{s:intersect basis}

In this section, we compute the intersection numbers $\alpha\cdot D$, $\alpha\cdot Y$, $\beta\cdot D$, and $\beta\cdot Y$, where $\alpha$ and $\beta$ are the test curves defined in Section~\ref{s:test curves}, and $D$ and $Y$ are the divisor classes which form a $\ZZ$-basis for $\Pic R$ described in the Introduction.

Recall that if  \[
 0 \to \cA \to V_{R\times \PP^1} \to \BB \to 0 \]
is the universal short exact sequence on $R\times \PP^1$, and $h=c_1\bigl(\pi_2^*\OO_{\PP^1}(1)\bigr)$, then  \[
 Y={\pi_1}_*\bigl(h\cdot c_1(\BB)\bigr),\quad D={\pi_1}_*\bigl(c_2(\BB)\bigr). \]
By Whitney's formula, \[
 c_1(\BB)=-c_1(\cA), \quad c_2(\BB)=c_1(\cA)^2-c_2(\cA). \]
So we also have \[
 Y={\pi_1}_*\bigl(-h\cdot c_1(\cA)\bigr),\quad D={\pi_1}_*\bigl(c_1(\cA)^2-c_2(\cA)\bigr). \]
Recall that the curve $\alpha\colon C\to R$ is induced by an injective sheaf homomorphism $\varphi\colon A_S\to V_S$ on $S=C\times \PP^1$, where \[
 A_S=\bigoplus_{i=1}^k A_i^{\spcheck}(-m_i). \]
Hence 
\begin{align*}
\alpha\cdot Y&=-h\cdot c_1(A_S)=\sum_{i=1}^k a_i\,; \\
\alpha\cdot D&=c_1(A_S)^2-c_2(A_S)=2d\sum_{i=1}^k a_i - \sum_{1\le i\ne j \le k}a_im_j \\
&\phantom{MMMMMlllllll||||}{}=\begin{cases} 
  \displaystyle (d+d_1) \sum_{i=1}^k a_i, &\text{if $k\mid d$;} \\
  \displaystyle (d+d_1)\sum_{i=1}^{\ell_1}a_i + (d+d_1+1) \sum_{i=\ell_1+1}^k a_i, &\text{if $k\nmid d$,}
 \end{cases} 
\end{align*}
where $a_i=\deg A_i$ for all $i$, $d_1=\lfloor d/k \rfloor$, and $\ell_1=k(d_1+1)-d$. 

The computation for the other curve $\beta$ is similar. Recall that $\beta\colon C\to R$ is induced by a surjective sheaf homomorphism $\psi\colon V_S\to B_S$ on $S=C\times \PP^1$, where \[
 B_S=\bigoplus_{i=1}^r B_i(n_i). \]
Hence 
\begin{align*}
\beta\cdot Y&=h\cdot c_1(B_S)=\sum_{i=1}^r b_i\,; \\
\beta\cdot D&=c_2(B_S)=\sum_{1\le i\ne j \le r}b_in_j 
 =\begin{cases} 
  \displaystyle (d-d_2) \sum_{i=1}^r b_i, &\text{if $r\mid d$;} \\
  \displaystyle (d-d_2)\sum_{i=1}^{\ell_2}b_i + (d-d_2-1) \sum_{i=\ell_2+1}^r b_i, &\text{if $r\nmid d$,}
 \end{cases} 
\end{align*}
where $b_i=\deg B_i$ for all $i$, $d_2=\lfloor d/r \rfloor$, and $\ell_2=r(d_2+1)-d$. 

\section{Intersections of the test curves with $\Dunb$ and $\Ddeg$} \label{s:intersect generators}

In this section, we compute the intersection numbers $\alpha\cdot \Dunb$, $\alpha\cdot \Ddeg$, $\beta\cdot \Dunb$, and $\beta\cdot \Ddeg$, where $\alpha$ and $\beta$ are the test curves defined in Section~\ref{s:test curves}, and $\Dunb$ and $\Ddeg$ are the effective divisors described in the Introduction which are to be shown to span the effective cone of $R$. The answers are summarized in Proposition~\ref{p:intersect generators} at the end of the section.

Let  \[
 0 \to \cA \to V_{R\times \PP^1} \to \BB \to 0 \]
be the universal short exact sequence on $R\times \PP^1$. Recall that when $k\mid d$, we define  \[
 \Dunb=\{p\in R \mid \text{$\cA_p$ has unbalanced splitting (as a locally free sheaf on $\PP^1$)} \}. \]
So we see that $\alpha\cdot \Dunb=0$ in this case. If $k\nmid d$, then for all $p$ outside a codimension two locus in $R$, $\cA_p$ splits as a direct sum of $\ell_1$ copies of $\OO_{\PP^1}(-d_1)$ and $(k-\ell_1)$ copies of $\OO_{\PP^1}(-d_1-1)$, and the direct sum of the $\ell_1$ copies of $\OO_{\PP^1}(-d_1)$ is a distinguished subsheaf $\EE_p$ of $\cA_p$. On the curve $\alpha$ this corresponds to the subsheaf \[  E_S=\bigoplus_{i=1}^{\ell_1} A_i^{\spcheck}(-d_1) \]
of $A_S=E_S\oplus \bigoplus_{i=\ell_1+1}^{k} A_i^{\spcheck}(-d_1-1)$. Recall that $\Dunb$ in this case is defined by an incidence condition on the directrix $\PP(\EE_p)\to\PP(V)$, hence $\alpha\cdot \Dunb$ can be interpreted as $c_2\bigl(\coker(E_S \to V_S)\bigr)$, which equals $c_1(E_S)^2-c_2(E_S)$ by Whitney's formula. Thus \[
\alpha\cdot \Dunb= c_1(E_S)^2-c_2(E_S)=2d_1\ell_1\sum_{i=1}^{\ell_1}a_i-d_1(\ell_1-1)\sum_{i=1}^{\ell_1}a_i=d_1(\ell_1+1)\sum_{i=1}^{\ell_1}a_i. \]

Before turning to the computation of $\alpha\cdot \Ddeg$, let us elaborate on the construction of $\Ddeg$ itself a little bit more. Our definition of $\Ddeg$ was motivated by the following observation.

\begin{proposition} \label{p:Ddeg}
Let\/ $0\to A\to V_{\PP^1}\to B\to 0$ be a short exact sequence of vector bundles on $\PP^1$, where as always $V_{\PP^1}=\OO_{\PP^1}\otimes V$ is the trivial vector bundle of rank~$n$, $k=\rank A$, $r=\rank B$, and $d=\deg B$. If $d=r$, then the image of\/ $\PP(A)\to\PP(V)$ is degenerate (i.e. lies in some hyperplane) if and only if $B$ is unbalanced (i.e. not isomorphic to $\bigoplus_{i=1}^r \OO_{\PP^1}(1)$).
\end{proposition}

\begin{proof}
There is an induced short exact sequence $0\to B^{\spcheck}\to V^{\spcheck}_{\PP^1}\to A^{\spcheck}\to 0$ of the dual bundles, which in turn induces an exact sequence \[
0\to H^0(\PP^1,B^{\spcheck})\to H^0(\PP^1,V^{\spcheck}_{\PP^1})\to H^0(\PP^1,A^{\spcheck}) \]
of the spaces of global sections. Identifying $H^0(\PP^1,V^{\spcheck}_{\PP^1})$ with $V^{\spcheck}$, it follows that \[
H^0(\PP^1,B^{\spcheck})=\{f\in V^{\spcheck}\mid \text{$f=0$ on the bundle $A$}\}. \]
Since $d=r$, $B$ is unbalanced if and only if $H^0(\PP^1,B^{\spcheck})\ne 0$, namely there exists a nonzero $f\in V^{\spcheck}$ which is zero on the bundle $A$, or equivalently, the image of $\PP(A)\to\PP(V)$ lies in the hyperplane $f=0$.
\end{proof}

As mentioned in the Introduction, Proposition~\ref{p:Ddeg} suggests defining $\Ddeg$ in general to be ``the locus where $B$ is unbalanced''. To make this precise, first suppose $r\ge 2$. Recall that the open subset $R^0$ of $R$ which parametrizes degree~$d$, rank~$r$ quotient bundles of $V_{\PP^1}$ has an irreducible complement $R\setminus R^0$ of codimension $r$ \cite[Corollary~3.3.8]{Shao}. Since $r\ge 2$, defining $\Ddeg$ on $R$ is thus equivalent to defining it on $R^0$. If $r\mid d$, we simply define $\Ddeg$ on $R^0$ to be \[
 \Ddeg=\{p\in R^0 \mid \text{$\BB_p$ has unbalanced splitting (as a locally free sheaf on $\PP^1$)} \}. \]
If $r\nmid d$, set $d=rd_2+(r-\ell_2)$ for $0<\ell_2<r$. Then for all $p\in R^0$ outside a codimension two subset, $\BB^{\spcheck}_p$ splits as a direct sum of $\ell_2$ copies of $\OO_{\PP^1}(-d_2)$ and $(r-\ell_2)$ copies of $\OO_{\PP^1}(-d_2-1)$. In this situation the direct sum of the $\ell_2$ copies of $\OO_{\PP^1}(-d_2)$ is a distinguished subsheaf $\FF_p$ of $\BB^{\spcheck}_p$, whose projectivization $\PP(\FF_p)\to \PP(V^{\spcheck})$ is the directrix of the scroll $\PP(\BB^{\spcheck}_p)\to \PP(V^{\spcheck})$. We define $\Ddeg$ to be the closure in $R$ of the locus of $p$ such that the image of $\PP(\FF_p)\to \PP(V^{\spcheck})$ meets a fixed $(n-2-\ell_2)$-dimensional subspace in $\PP(V^{\spcheck})$.

We now begin to compute $\alpha\cdot\Ddeg$ in the case $r\ge 2$ and $r\mid d$.  Recall that the curve $\alpha\colon C\to R$ is induced by an injective sheaf homomorphism \[
\varphi\colon A_S=\bigoplus_{j=1}^k A_j^{\spcheck}(-m_j)\longrightarrow V_S \]
on $S=C\times \PP^1$. Choosing a basis $v_1,\ldots,v_n$ for $V$ and a basis $x,y$ for $H^0\bigl(\PP^1,\OO_{\PP^1}(1)\bigr)$, we can express $\varphi$ as an $n\times k$ matrix \[
 \varphi=\biggl(\sum_{t=0}^{m_j}s_{ij,t}\,x^ty^{m_j-t}\biggr)_{\substack{1\le i\le n \\ 1\le j \le k}} \]
where $s_{ij,t}\in H^0(C,A_j)$. We choose the line bundles $A_j$'s to be sufficiently positive and the sections $s_{ij,t}$'s to be sufficiently general so that the cokernel of $\varphi$ is locally free, which can be done since $r\ge 2$. In other words, we choose the curve $\alpha\colon C\to R$ to lie in $R^0$, so that the intersection number $\alpha\cdot\Ddeg$ can be computed geometrically. Let $B_S=\coker\varphi$. (Note that, unlike the $B_S$ which defines the curve $\beta$ in Section~\ref{s:test curves}, the $B_S=\coker\varphi$ here is not a direct sum of line bundles on $S$.) Then $\alpha\cdot\Ddeg$ equals the number of points $p\in C$ such that $B_{S,p}=B_S|_{\{p\}\times \PP^1}$ has unbalanced splitting, or equivalently $H^0\bigl(\PP^1,B^{\spcheck}_{S,p}(d_2-1)\bigr)\ne 0$ where $d_2=d/r$. Applying a similar argument as in the proof of Proposition~\ref{p:Ddeg} to the short exact sequence  \[
 0\to A_{S,p} \xrightarrow{\varphi_p} V_{S,p} \to B_{S,p} \to 0 \]
of vector bundles on $\{p\}\times \PP^1$, we see that  \[
 H^0\bigl(\PP^1,B^{\spcheck}_{S,p}(d_2-1)\bigr)=\bigl\{f\in H^0\bigl(\PP^1,\OO_{\PP^1}(d_2-1)\bigr)\otimes V^{\spcheck} \bigm| \text{$f=0$ on $A_{S,p}$}\bigr\}. \]
Expressing $f$ as a row vector \[
 f=\biggl( \sum_{t=0}^{d_2-1}f_{1,t}\,x^ty^{d_2-1-t},\ldots,\sum_{t=0}^{d_2-1}f_{n,t}\,x^ty^{d_2-1-t}\biggr) \]
with respect to the dual basis of $v_1,\ldots,v_n$, the condition that $f=0$ on $A_{S,p}$ is equivalent to the matrix product $f\cdot\varphi_p$ being $0$, where $\varphi_p$ denotes the matrix $\varphi$ with every $s_{ij,t}$ evaluated at $p$. For each fixed $p\in C$, $f\cdot\varphi_p=0$ can be viewed as a system of linear equations in $nd_2$ unknowns $f_{i,t}$, $1\le i\le n$, $0\le t\le d_2-1$. Precisely, for each $j\in\{1,\ldots,k\}$ and each $t\in\{0,\ldots,m_j+d_2-1\}$ there is a linear equation \[
 \sum_{\substack{t_1+t_2=t\\ 1\le i\le n}} s_{ij,t_1}(p)\cdot f_{i,t_2}=0. \]
The number of equations is thus \[
 kd_2+\sum_{j=1}^k m_j=kd_2+d=kd_2+rd_2=nd_2, \]
which is the same as the number of unknowns. Hence the existence of a nonzero solution $f$ is equivalent to the vanishing of the determinant of the coefficient matrix of the linear system. This determinant is a section of the line bundle \[
 A_1^{\otimes(m_1+d_2)}\otimes\cdots\otimes A_k^{\otimes(m_k+d_2)} \]
on $C$. Therefore \[
\alpha\cdot\Ddeg=\sum_{i=1}^k a_i(m_i+d_2)=\begin{cases} 
  \displaystyle (d_1+d_2) \sum_{i=1}^k a_i, &\text{if $k\mid d$;} \\
  \displaystyle (d_1+d_2)\sum_{i=1}^{\ell_1}a_i + (d_1+d_2+1) \sum_{i=\ell_1+1}^k a_i, &\text{if $k\nmid d$.}
 \end{cases}  \]
This is the formula for $\alpha\cdot\Ddeg$ when $r\mid d$ and $r\ge 2$. Incidentally, the same formula still holds when $r=1$. Recall from the Introduction that when $r=1$, $\Ddeg$ is defined to be $R\setminus R^0$, which is an irreducible divisor by   \cite[Corollary~3.3.8]{Shao}. It turns out that its divisor class is equal to $D$ (Proposition~\ref{p:Ddeg when r=1}), and one readily checks that the above formula with $r=1$ (namely $d_2=d$) gives the same answer as the formula for $\alpha\cdot D$ computed in Section~\ref{s:intersect basis}.

\begin{proposition} \label{p:Ddeg when r=1}
If\/ $r=1$, then $\Ddeg=R\setminus R^0$ is an irreducible divisor whose divisor class is equal to $D={\pi_1}_*\bigl(c_2(\BB)\bigr)$.
\end{proposition}

\begin{proof}
Since $r=1$, $\BB|_{R^0}$ is a line bundle and hence $c_2(\BB)|_{R^0}=0$. It follows that the divisor class $D={\pi_1}_*\bigl(c_2(\BB)\bigr)$ can be represented by some divisor supported on $\Ddeg=R\setminus R^0$, thus $D=c\Ddeg$ for some integer $c$ since $R\setminus R^0$ is irreducible by \cite[Corollary~3.3.8]{Shao}. Recall from the Introduction that $D$ and $Y$ form a $\ZZ$-basis for $\Pic R$, hence $c$ must be $\pm 1$. To conclude that $c=1$, simply observe that $\alpha\cdot D>0$ by the formula in Section~\ref{s:intersect basis}, and $\alpha\cdot \Ddeg\ge 0$ since $\alpha$ does not lie in $\Ddeg$. 
\end{proof}

Note that if we want to compute $\beta\cdot\Dunb$ when $k\mid d$, the situation is completely dual to the computation we just did for $\alpha\cdot\Ddeg$ when $r\mid d$. So without extra work we have \[
\beta\cdot\Dunb=\sum_{i=1}^r b_i(n_i+d_1)=\begin{cases} 
  \displaystyle (d_1+d_2) \sum_{i=1}^r b_i, &\text{if $r\mid d$;} \\
  \displaystyle (d_1+d_2)\sum_{i=1}^{\ell_2}b_i + (d_1+d_2+1) \sum_{i=\ell_2+1}^r b_i, &\text{if $r\nmid d$.}
 \end{cases}  \]
Similarly, the computation of $\beta\cdot\Ddeg$ when $r\ge 2$ is dual to the computation of $\alpha\cdot\Dunb$ that was done in the beginning of this section. If $r=1$ then obviously $\beta\cdot\Ddeg=0$, since $\beta$ lies in $R^0$ while $\Ddeg=R\setminus R^0$ in this case. Hence for all $r>0$ we have \[
\beta\cdot\Ddeg=\begin{cases} 
  0, &\text{if $r\mid d$;} \\
  \displaystyle d_2(\ell_2+1)\sum_{i=1}^{\ell_2}b_i, &\text{if $r\nmid d$.}
 \end{cases}  \] 

Let us summarize the intersection numbers that were computed in this section:

\begin{proposition} \label{p:intersect generators}
Let $\alpha$ and $\beta$ be the test curves defined in Section~\ref{s:test curves}. Let $d_1=\lfloor d/k \rfloor$, $d_2=\lfloor d/r \rfloor$, and let $\ell_1$, $\ell_2$, $a_i$, and $b_i$ be as defined in Section~\ref{s:test curves}. We have
\begin{enumerate}
\item $\alpha\cdot\Dunb=\begin{cases} 
  0, &\text{if $k\mid d$;} \\
  \displaystyle d_1(\ell_1+1)\sum_{i=1}^{\ell_1}a_i, &\text{if $k\nmid d$.}
 \end{cases}$
\item If $k\mid d$, then \[
\beta\cdot\Dunb=\begin{cases} 
  \displaystyle (d_1+d_2) \sum_{i=1}^r b_i, &\text{if $r\mid d$;} \\
  \displaystyle (d_1+d_2)\sum_{i=1}^{\ell_2}b_i + (d_1+d_2+1) \sum_{i=\ell_2+1}^r b_i, &\text{if $r\nmid d$.}
 \end{cases}  \]
\item If $r\mid d$, then \[
\alpha\cdot\Ddeg=\begin{cases} 
  \displaystyle (d_1+d_2) \sum_{i=1}^k a_i, &\text{if $k\mid d$;} \\
  \displaystyle (d_1+d_2)\sum_{i=1}^{\ell_1}a_i + (d_1+d_2+1) \sum_{i=\ell_1+1}^k a_i, &\text{if $k\nmid d$.}
 \end{cases}  \]
\item $\beta\cdot\Ddeg=\begin{cases} 
  0, &\text{if $r\mid d$;} \\
  \displaystyle d_2(\ell_2+1)\sum_{i=1}^{\ell_2}b_i, &\text{if $r\nmid d$.}
 \end{cases}$ 
\end{enumerate}
\end{proposition}
We omit the computation of $\beta\cdot\Dunb$ when $k\nmid d$ and $\alpha\cdot\Ddeg$ when $r\nmid d$ because we will not need them.

\section{Proof of the theorem when $r>0$} \label{s:proof r>0}

In this section, we prove Theorem~\ref{t:main} when $r>0$. The case $r=0$ will be treated in the next section.

First we find the divisor classes of $\Dunb$ and $\Ddeg$ in terms of the basis $D$ and $Y$. Starting with $\Dunb$, we want to determine the integers $e_1$ and $e_2$ such that \[
    \Dunb=e_1 D+ e_2 Y, \]
using the intersection numbers computed in Section~\ref{s:intersect basis} and Section~\ref{s:intersect generators}. We handle the cases $k\mid d$ and $k\nmid d$ separately:

\begin{itemize}
\item If $k\nmid d$, we claim that $\Dunb=d_1(\ell_1+1)\bigl(-D+(d+d_1+1)Y\bigr)$. This can be seen by intersecting $\alpha$ with $\Dunb=e_1 D+ e_2 Y$, which gives 
\begin{align*}
d_1(\ell_1+1)\sum_{i=1}^{\ell_1}a_i&=e_1\Bigl((d+d_1)\sum_{i=1}^{\ell_1}a_i + (d+d_1+1) \sum_{i=\ell_1+1}^k a_i\Bigr)+e_2\Bigl(\sum_{i=1}^k a_i\Bigr) \\
&=\bigl(e_1(d+d_1)+e_2\bigr)\sum_{i=1}^{\ell_1}a_i+\bigl(e_1(d+d_1+1)+e_2\bigr)\sum_{i=\ell_1+1}^k a_i.
\end{align*}
As this equality holds for all sufficiently large $a_i$'s, one sees that \[
 e_1=-d_1(\ell_1+1), \quad e_2=d_1(\ell_1+1)(d+d_1+1). \]

\item If $k\mid d$, we claim that $\Dunb=-D+(d+d_1)Y$. If $r\nmid d$, one can see this by intersecting $\beta$ with $\Dunb=e_1 D+ e_2 Y$, which gives 
\begin{align*}
(d_1+d_2)\sum_{i=1}^{\ell_2}b_i + (&d_1+d_2+1) \sum_{i=\ell_2+1}^r b_i \\
&=e_1\Bigl((d-d_2)\sum_{i=1}^{\ell_2}b_i + (d-d_2-1) \sum_{i=\ell_2+1}^r b_i\Bigr)+e_2\Bigl(\sum_{i=1}^r b_i\Bigr) \\
&=\bigl(e_1(d-d_2)+e_2\bigr)\sum_{i=1}^{\ell_2}b_i+\bigl(e_1(d-d_2-1)+e_2\bigr)\sum_{i=\ell_2+1}^r b_i.
\end{align*}
As this equality holds for all sufficiently large $b_i$'s, one sees that $e_1=-1$ and $e_2=d+d_1$. If $r\mid d$, intersecting $\alpha$ with $\Dunb=e_1 D+ e_2 Y$ yields \[
0=e_1(d+d_1) \sum_{i=1}^k a_i+e_2\sum_{i=1}^k a_i, \]
while intersecting $\beta$ with $\Dunb=e_1 D+ e_2 Y$ yields \[
(d_1+d_2) \sum_{i=1}^r b_i=e_1(d-d_2) \sum_{i=1}^r b_i+e_2\sum_{i=1}^r b_i. \]
So together they still imply $e_1=-1$ and $e_2=d+d_1$.
\end{itemize}
The divisor class of $\Ddeg$ can be found with similar argument: 
\begin{itemize}
\item If $r\nmid d$, one sees that \[
\Ddeg=d_2(\ell_2+1)\bigl(D+(-d+d_2+1)Y\bigr) \]
by intersecting with the curve $\beta$.
\item If $r\mid d$, one sees that \[
\Ddeg=D+(-d+d_2)Y \]
by intersecting with the curve $\alpha$ if $k\nmid d$, or intersecting with both $\alpha$ and $\beta$ if $k\mid d$.
\end{itemize}
Thus the expressions of the classes of $\Dunb$ and $\Ddeg$ in terms of the basis $D$ and $Y$ in Theorem~\ref{t:main} are verified. 

It remains to prove that $\Dunb$ and $\Ddeg$ span the effective cone. The usual strategy to show that an effective divisor lies on the boundary of the effective cone is to check that it has zero intersection with some moving curve. That was the method Coskun-Starr used in \cite{CS}, where, since they were dealing with the case $d=r$, it is not hard to see that all generic curves are conjugate under the action of $\PP GL(V)$ and are thus moving curves (\cite[Lemma~2.4]{CS}). Unfortunately we were unable to see how to adapt the argument to handle the case $d>r$. So we resort to the following lemma instead, which basically says that since the rank of $\Pic R$ is only two, it is enough to \emph{pointwisely} find a curve which has zero intersection with $\Dunb$ (resp. $\Ddeg$). 

\begin{lemma} \label{l:span Eff}
Suppose two effective divisors $D_1$ and $D_2$ on $R$ are linearly independent in $\Pic R$. If there is an open subset $U$ of $R$ such that for each point $p\in U$, there exists a curve $\gamma_i$ (depending on $p$) in $R$ passing through $p$ and $\gamma_i \cdot D_i=0$, $i=1,2$, then $D_1$ and $D_2$ span the effective cone of $R$.
\end{lemma}

\begin{proof} 
Let $E$ be an arbitrary effective divisor on $R$ and write \[
  E=e_1 D_1 + e_2 D_2, \quad e_1,e_2\in\QQ.  \]
Pick a point $p\in U$ such that $p\notin E\cup D_2$, and let $\gamma_1$ be a curve through $p$ such that $\gamma_1\cdot D_1=0$. Then $\gamma_1\cdot E\ge 0$ and $\gamma_1\cdot D_2> 0$, hence $e_2\ge 0$. Similarly $e_1\ge 0$.
\end{proof}

To apply the lemma, we choose the open subset $U$ of $R$ to be \[
U=\{p\in R^0 \mid \text{Both $\cA_p$ and $\BB_p$ have the most balanced splitting types possible}\}. \]
It is not hard to see that given any $p\in U$, there exist curves $\alpha$ and $\beta$ of the forms described in Section~\ref{s:test curves} which pass through $p$ and satisfy $\alpha\cdot\Dunb=0$ and $\beta\cdot\Ddeg=0$. For example, as done in Section~\ref{s:intersect generators}, after choosing a basis for $V$ and for $H^0\bigl(\PP^1,\OO_{\PP^1}(1)\bigr)$, we can represent a curve of type $\alpha$ by an $n \times k$ matrix of the form \[
 \varphi=\biggl(\sum_{t=0}^{m_j}s_{ij,t}\,x^ty^{m_j-t}\biggr)_{\substack{1\le i\le n \\ 1\le j \le k}} \]
where $s_{ij,t}\in H^0(C,A_j)$. A point $p\in U$ can be represented by a similar matrix \[
 \biggl(\sum_{t=0}^{m_j}p_{ij,t}\,x^ty^{m_j-t}\biggr)_{\substack{1\le i\le n \\ 1\le j \le k}} \]
where $p_{ij,t}$ are constants. In order for the curve $\varphi$ to pass through $p$, it is enough that at some point in $C$, the ratio of the sections $s_{ij,t}$ in each column coincides with the corresponding ratio given by the constants $p_{ij,t}$. And one sees from Proposition~\ref{p:intersect generators} that $\alpha\cdot\Dunb$ is always $0$ if $k\mid d$, while if $k\nmid d$ then we take $a_1=\cdots=a_{\ell_1}=0$. Although this would require taking the line bundles $A_1,\ldots,A_{\ell_1}$ to be trivial, the other line bundles $A_{\ell_1+1},\ldots,A_k$ and their sections could still be chosen general enough so that all the previous intersection computations went through. Thus we can conclude from Lemma~\ref{l:span Eff} that $\Dunb$ and $\Ddeg$ span the effective cone of $R$.

\section{Proof of the theorem when $r=0$} \label{s:proof r=0}

The case $r=0$ of Theorem~\ref{t:main} requires separate treatment for two reasons: the divisor $\Ddeg$ is defined in a totally different manner, and the test curve $\beta$ defined in Section~\ref{s:test curves} no longer exists. We do still have the test curve $\alpha$, as well as the formula for $\alpha\cdot\Dunb$ in Proposition~\ref{p:intersect generators}~(1). Let us begin by computing $\alpha\cdot\Ddeg$. Let \[
\varphi\colon A_S=\bigoplus_{j=1}^k A_j^{\spcheck}(-m_j)\longrightarrow V_S \]
be the injective sheaf homomorphism on $S=C\times \PP^1$ that induces the curve $\alpha\colon C\to R$. Choosing a basis $v_1,\ldots,v_n$ for $V$ and a basis $x,y$ for $H^0\bigl(\PP^1,\OO_{\PP^1}(1)\bigr)$, we can express $\varphi$ as an $n\times n$ matrix \[
 \varphi=\biggl(\sum_{t=0}^{m_j}s_{ij,t}\,x^ty^{m_j-t}\biggr)_{1\le i,j\le n} \]
where $s_{ij,t}\in H^0(C,A_j)$. Let $B_S=\coker\varphi$. By the definition of $\Ddeg$, $\alpha\cdot\Ddeg$ equals the number of points $p\in C$ such that the support of $B_{S,p}=B_S|_{\{p\}\times \PP^1}$ does not consist of $d$ distinct points. The support of $B_{S,p}$ is precisely the set of points on $\PP^1$ whose homogeneous coordinates $(x:y)$ satisfy $\det \varphi_p=0$, where $\varphi_p$ denotes the matrix $\varphi$ with every $s_{ij,t}$ evaluated at $p$. Hence the support of $B_{S,p}$ does not consist of $d$ distinct points if and only if the discriminant of $\det\varphi_p$ vanishes. The discriminant of $\det\varphi$ is a section of the line bundle $(A_1\otimes\cdots\otimes A_k)^{\otimes 2(d-1)}$ on $C$. Therefore \[
\alpha\cdot\Ddeg=2(d-1)\sum_{i=1}^k a_i. \]

Next we will show that, given any point in $R\setminus \Ddeg$, there exists a curve $\gamma\colon C\to R$ passing through it such that \[
  \gamma\cdot\Ddeg=\gamma\cdot Y= 0. \]
Let $B$ be a degree~$d$, rank~$0$ quotient sheaf of $V_{\PP^1}$ which corresponds to a point in $R\setminus \Ddeg$, namely $\Supp B$ is a set of $d$ distinct points. We will construct the curve $\gamma$ by deforming $B$ into a family. In fact, we will only deform the stalk of $B$ at a single point in $\Supp B$. So let $p$ be an arbitrary point in $\Supp B$, and denote by $(\OO_p,\mm_p)$ the local ring of $\PP^1$ at $p$. Let $V_p$ and $B_p$ be the stalks of the sheaves $V_{\PP^1}$ and $B$ at $p$, respectively, and let $A_p$ be the kernel of $V_p\to B_p$. Since $\Supp B$ is a set of $d$ distinct points, $\bigwedge^n A_p$ has colength one in $\bigwedge^n V_p$. Hence there exists an $\OO_p$-basis $\nu_1,\ldots,\nu_n$ of $V_p$ such that \[
 A_p=\mm_p \nu_1 \oplus \OO_p \nu_2 \oplus \cdots \oplus \OO_p \nu_n. \]
Let $C=\PP^1$ with homogeneous coordinates $(z_0:z_1)$. Consider the following family of $\OO_p$-submodules of $V_p$ over $C$: \[
A_p(z_0:z_1)=\begin{cases}
 \mm_p \nu_1 \oplus \OO_p (z_1\nu_1+z_0\nu_2) \oplus \OO_p \nu_3 \oplus \cdots \oplus \OO_p \nu_n, &\text{if $z_0\ne 0$;} \\
 \mm_p \nu_2 \oplus \OO_p (z_1\nu_1+z_0\nu_2) \oplus \OO_p \nu_3 \oplus \cdots \oplus \OO_p \nu_n, &\text{if $z_1\ne 0$.}
\end{cases} \]
This is well-defined, since when both $z_0$ and $z_1$ are nonzero, the two expressions give the same submodule. Let $B_p(z_0:z_1)=V_p/ A_p(z_0:z_1)$, and let $B(z_0:z_1)$ be the sheaf on $\PP^1$ whose stalk at $p$ is $B_p(z_0:z_1)$, and at every other point is just the same as the stalk of $B$. Then $B(z_0:z_1)$ induces a curve $\gamma\colon C\to R$ which maps $(1:0)\in C$ to the point in $R$ corresponding to $B$. Note that $\Supp B(z_0:z_1)=\Supp B$ for all $(z_0:z_1)$. This implies that $\gamma$ does not meet $\Ddeg$, hence $\gamma\cdot\Ddeg=0$. It also implies that, viewed as a sheaf on the surface $S=C\times \PP^1$, $B(z_0:z_1)$ (and hence also its first Chern class) is completely supported on $C\times \Supp B$, thus $\gamma\cdot Y=0$.

We now proceed to find the divisor classes of $\Dunb$ and $\Ddeg$ in terms of the basis $D$ and $Y$. If $k\nmid d$, this can be done using the curve $\alpha$ alone: the class of $\Dunb$ has been worked out in Section~\ref{s:proof r>0} to be \[
  \Dunb=d_1(\ell_1+1)\bigl(-D+(d+d_1+1)Y\bigr). \]
As for $\Ddeg$, writing $\Ddeg=e_1D+e_2Y$ and intersecting with $\alpha$ gives 
\begin{align*}
2(d-1)\sum_{i=1}^k a_i&=e_1\Bigl((d+d_1)\sum_{i=1}^{\ell_1}a_i + (d+d_1+1) \sum_{i=\ell_1+1}^k a_i\Bigr)+e_2\Bigl(\sum_{i=1}^k a_i\Bigr) \\
&=\bigl(e_1(d+d_1)+e_2\bigr)\sum_{i=1}^{\ell_1}a_i+\bigl(e_1(d+d_1+1)+e_2\bigr)\sum_{i=\ell_1+1}^k a_i.
\end{align*}
It follows that $e_1=0$, $e_2=2(d-1)$, and hence $\Ddeg=2(d-1)Y$. If $k\mid d$, we can still deduce that $\Ddeg=2(d-1)Y$ with the extra help from the new curve $\gamma$, for intersecting $\alpha$ with $\Ddeg=e_1D+e_2Y$ now gives \[
2(d-1)\sum_{i=1}^k a_i=e_1(d+d_1)\sum_{i=1}^{k}a_i +e_2\sum_{i=1}^k a_i=\bigl(e_1(d+d_1)+e_2\bigr)\sum_{i=1}^{k}a_i, \]
which is not enough to determine $e_1$ and $e_2$. So we need the additional information from intersecting $\gamma$ with $\Ddeg=e_1D+e_2Y$, which turns out to be $e_1=0$ since $\gamma\cdot\Ddeg=\gamma\cdot Y=0$. Hence $e_2=2(d-1)$, and $\Ddeg=2(d-1)Y$ still holds. As for $\Dunb$, writing $\Dunb=e_1D+e_2Y$ and intersecting with $\alpha$ gives \[
0=e_1(d+d_1)\sum_{i=1}^{k}a_i +e_2\sum_{i=1}^k a_i=\bigl(e_1(d+d_1)+e_2\bigr)\sum_{i=1}^{k}a_i, \]
which implies that  \[
  \Dunb=c_1\bigl(-D+(d+d_1)Y\bigr) \]
for some nonzero constant $c_1$. Unfortunately since we do not know how to compute $\gamma\cdot \Dunb$ and $\gamma\cdot D$, we do not know the value of $c_1$. However we can at least see that $c_1>0$ once we know that $\Dunb$ and $\Ddeg$ span the effective cone, since (the closure of) the effective cone has to contain the nef cone (cf. Figure~\ref{f:cones} right).

It remains to show that $\Dunb$ and $\Ddeg$ span the effective cone. For this we use Lemma~\ref{l:span Eff}, taking the open subset $U$ of $R$ to be \[
 U=\{p\in R\mid \text{$\cA_p$ has the most balanced splitting type possible}\}\setminus \Ddeg. \]
As we have seen, given any $p\in U$, there exist curves $\alpha$ and $\gamma$ which pass through $p$ and satisfy $\alpha\cdot\Dunb=0$ and $\gamma\cdot\Ddeg=0$. Hence $\Dunb$ and $\Ddeg$ span the effective cone of $R$ by Lemma~\ref{l:span Eff}.


\end{document}